\documentclass[12pt]{amsart}

\usepackage{amssymb}
\usepackage{amsmath}
\usepackage{amscd}
\usepackage{bbm} 

\usepackage{diagbox}

\usepackage{hyperref}
\usepackage{cleveref}

\numberwithin{equation}{section}
\newtheorem{theorem}{Theorem}
\newtheorem{proposition}[theorem]{Proposition}
\newtheorem{lemma}[theorem]{Lemma}

\theoremstyle{definition}

\newtheorem{example}[theorem]{Example}
\newtheorem{remark1}[theorem]{Remark}
\newtheorem{openproblem1}[theorem]{Open problem}

\newenvironment{remark}{\begin{remark1}\rm}{\end{remark1}}

\numberwithin{theorem}{section}

\newcommand{\C}{\mathbb{C}}

\newcommand{\R}{\mathbb{R}}

\newcommand{\cA}{\mathcal{A}}
\newcommand{\cB}{\mathcal{B}}

\newcommand{\cX}{\mathcal{X}}

\DeclareMathOperator{\diag}{diag}

\DeclareMathOperator{\tr}{tr}

\DeclareMathOperator{\pos}{Pos}
\DeclareMathOperator{\den}{D}
\DeclareMathOperator{\her}{Herm}

\newcommand{\domain}{\bigoplus_{i=1}^N L(\mathcal{A}_i)}

\DeclareMathOperator{\id}{\mathbb{I}}

\begin{document}
\title[Semidefinite network games]{Semidefinite network games: multiplayer minimax and  complementarity problems }

\author[C.~Ickstadt]{Constantin Ickstadt}

\author[T.~Theobald]{Thorsten Theobald}

\address{Constantin Ickstadt, Thorsten Theobald:
	Goethe-Universit\"at, FB 12 -- Institut f\"ur Mathematik,
	Postfach 11 19 32, 60054 Frankfurt am Main, Germany}

\email{ickstadt@math.uni-frankfurt.de,
	theobald@math.uni-frankfurt.de}

\author[E.~Tsigaridas]{Elias Tsigaridas}

\address{Elias Tsigaridas: Sorbonne Universit\'e, 
    Paris University, CNRS and Inria Paris.
	IMJ-PRG,  4 place Jussieu,
	75252 Paris Cedex 05, France}
	
\email{elias.tsigaridas@inria.fr}

\author[A.~Varvitsiotis]{Antonios Varvitsiotis}

\address{Singapore University of Technology and Design (ESD Pillar), 8 Somapah Road 487372, Singapore}

\email{antonios@sutd.edu.sg}




\begin{abstract}
Network games provide a powerful framework for modeling agent interactions 
in networked systems, where players are represented by nodes in a graph and 
their payoffs depend on the actions taken by their neighbors.   
Extending the framework of network games, we introduce and study 
semidefinite network games. In this model, each player selects a 
positive semidefinite matrix with trace equal to one, known as a 
density matrix, 
to engage in a two-player game with every neighboring node.	
The player's payoff is the cumulative payoff acquired from these edge games.
Initially, we focus on the zero-sum setting, where the sum of all players'
payoffs is equal to zero. We establish that, in this class of games, Nash
equilibria can be characterized as the projection of a spectrahedron.
Furthermore, we show that determining whether a semidefinite network game 
is a zero-sum game is equivalent to deciding if the value of a semidefinite
program is zero. Beyond the zero-sum case, we characterize Nash equilibria
as the solutions of a  semidefinite linear complementarity problem.
\end{abstract}
	
\maketitle

	\section{Introduction}
	\label{sec:intro}
	In normal form games, agents use probability distributions to select one action from a finite set, while their utility functions are  multilinear. In a broader context, continuous games enable agents to choose from an infinite set of pure strategies, often constrained to a compact set. Computationally, continuous games involve calculations with polynomial roots instead of just manipulating rational numbers, and their solvability relies on algorithmic techniques that go beyond linear programming.

	The network version of  normal form  games was introduced
by Janovskaja under the
name \emph{polymatrix game} \cite{Janovskaja-polymatrix-68}
and the connection to linear programming first appeared in \cite{BreFok-Opt-87,BreFok-Opt-98}.
A polymatrix game is called \emph{zero-sum} if, for each mixed strategy profile of the players, the sum of
all players' payoffs is zero.
Br\`egman and Fokin \cite{BreFok-Opt-98}
transform the zero-sum polymatrix game to a two-player zero-sum polyhedral game, which in turn
corresponds to an exponential size linear program.
Daskalakis and Papadimitriou \cite{DasPap-network-NE-09} consider the subclass of pairwise zero-sum polymatrix games, 
where each edge corresponds to a zero-sum game. To compute the Nash equilibria
they carefully round the optimal solution of a linear program.
Cai and Daskalakis \cite{CaiDas-multi-NE-11}
 among other results, directly relate
the equilibria of any polymatrix zero-sum game to a linear program.
Cai et al. \cite{ccdp-polymatrix-16} show how to compute the equilibria of a zero-sum
polymatrix game using a polynomial size linear program and how to recognize
in polynomial time whether a polymatrix game is zero-sum. 
They also show, using a nontrivial transformation, that
for every zero-sum polymatrix game there exists a
payoff-equivalent polymatrix game with constant-sum edges
and thus there exists a polymatrix game with zero-sum edges
which has the same set of Nash equilibria.
That work is one
of the two points of departure of our work.

Our second point of departure is the class of \emph{semidefinite games}
\cite{sdp-games}. 
In these games, the strategy set of each
player consists of a symmetric or Hermitian 
positive semidefinite matrix with trace one, commonly
referred to as density matrix, and the payoff
functions are bilinear; see 
Section \ref{su:sdpgames} for a formal definition. 
	Semidefinite games  fall within the class  of continuous games where the strategy space is a convex compact set and the payoffs are concave functions (in this case multilinear),
see, \cite{continuous,rosen}.
Semidefinite games can also be seen as the simplest
model of a two-player quantum game (see Section~\ref{se:quantum-games}).

For two-player zero-sum semidefinite games, we can efficiently compute the optimal strategies using semidefinite programming, 
 see \cite{jain-watrous-2009,sdp-games}.
 Moreover, \cite{sdp-games} establishes a strong connection—amounting to an almost equivalence—between two-player zero-sum semidefinite games and semidefinite programming, thus extending Dantzig’s seminal work on the near equivalence between bimatrix games and linear programming. This work,
 that is \cite{sdp-games}, also provides constructions of two-player semidefinite games with multiple Nash equilibria. In contrast, for semidefinite games with more than two players, \cite{sdp-games} does not offer new information beyond noting some immediate consequences derived from broader, more well-established classes of games.

Building on the concepts of polymatrix games and semidefinite games, we introduce and study a multiplayer class of games called \emph{semidefinite network games}, where players are interconnected through an undirected graph $G$ and each edge represents a two-player semidefinite game. 
The following table explains that we complete the square.
The analogy between normal form (simplex-based) and semidefinite games, under both global and networked interaction structures, is summarized in the table.

\begin{table}[h]
\centering
\begin{tabular}{|c|c|c|}
\hline
\diagbox{Strategies}{Interaction} & Global & Network structure \\
\hline
Simplex & Normal form games & Polymatrix games \\
\hline
Density matrices & Semidefinite games & This work \\
\hline
\end{tabular}
\end{table}

 A formal definition of semidefinite network games appears in \Cref{su:sdp-network-games}.
In this setting, each player selects a single density matrix, that she uses to play all the games associated with her neighboring nodes,
and her payoff is the cumulative payoff acquired from these edge games.

Since each player's strategy space is convex and compact, an equilibrium point always exists (see, e.g., \cite{Glicksberg-fp-1952}). However, beyond this existence result, addressing further relevant questions in semidefinite network games requires the development of novel methods.
We establish the following three~results:
	\begin{enumerate}
		\item In a zero-sum semidefinite network game, we can efficiently compute 
		Nash equilibria using  semidefinite programming.
		\item  A zero-sum semidefinite network game and, similarly, a constant-sum semidefinite network game, 
        can be recognized using semidefinite programming.
		\item Beyond the zero-sum case, semidefinite Nash equilibria are solutions to a semidefinite linear complementarity problem. 
	\end{enumerate} 
As is common in the study of network games (see, e.g., \cite{ccdp-polymatrix-16} for the classical setting), the zero-sum property in our context refers to the aggregate payoff across all players, rather than to the individual payoffs in each edge game.
Let us also mention that several related classes of continuous games involving semidefinite structures have already been investigated in the literature; see, for instance, \cite{mbn-2017,sfp-2014,scutari-palomar-2010}.

\section{Preliminaries \label{sec:prelim}}	

\subsection{Linear algebra.} First, we give a brief introduction to the linear algebraic notions that we use in our study; our notation follows closely \cite{watrous}.

  A complex Euclidean space is a finite-dimensional vector space over complex numbers (i.e.,  isomorphic to $\mathbb{C}^k$ for some $k\ge 1$)
 equipped with  the standard Euclidean inner product. 
We denote by $\cA = \C^{k}$ and we denote by $L(\cA)$  the space of linear maps $A: \cA \to \cA$. We use $A$ to represent  both the operator and its corresponding matrix representation, assuming a specific choice of coordinates.
	A Hermitian matrix  $A\in\her (\cA)$ is 
	\emph{positive semidefinite} (PSD), denoted $A\succeq 0$,  if all its eigenvalues are nonnegative. The set of PSD matrices acting on $\cA$  is denoted by $\pos(\cA)$. The set of density matrices acting on $\cA$ is the set of PSD matrices with trace equal to one, i.e., 
	\begin{equation}
		\label{eq:def-DA}
		\den(\cA)=\Big\{ X\in  \her(\cA): X\succeq 0 \text{ and }  \tr(X)=1 \Big\}.
	\end{equation}
	The set of density matrices is convex, compact, and its extreme points are rank-one density matrices, i.e., matrices of the form  $X=xx^\dagger$, where $\|x\|=1$.
	For any Hermitian  matrix $A\in \her (\cA)$ we have that 
	\begin{equation}\label{lambdamax}
		\begin{aligned}
			\lambda_{\max}(A)&=\max \{ \langle A, X\rangle : X\in \den(\cA) \}, \\
			\lambda_{\min}(A)& =\min \{ \langle A, X\rangle : X\in \den(\cA) \}.
		\end{aligned} 
	\end{equation}
	This well-known fact follows from the Rayleigh-Ritz characterization of the maximum (resp. minimum)  eigenvalue of a Hermitian matrix $A$, i.e., 
	$$\lambda_{\max}(A)=\max \{ x^\dagger A x: \|x\|=1\},$$ combined with the characterization
	of the extreme points of $\den(\cA)$ given above.
	Finally, for any $A\in\her(\cA)$ we repeatedly use that 
	$A\preceq t\id$ if and only if all the eigenvalues of $A$ are at most $ t$,
 where $\id$ is the identity matrix.

	Let $\cA = \C^{n}$ and $\cB = \C^m$ be arbitrary fixed complex Euclidean spaces. 
 We define the \emph{partial trace},
	$\tr_{\cA} = \tr \otimes \id_{\cB}$, as the map
	\begin{equation}
		\label{eq:partial-trace}
		\begin{array}{ccccl}
			\tr_{\cA} & : & L(\cA \otimes \cB) &\to& L(\cB) \\
			& & A\otimes B & \mapsto & \tr_{\cA}(A\otimes B) = \tr(A)\, B ,
		\end{array}
	\end{equation}
 where $\id_{\cB}$ is the identity matrix on $L(\cB).$
	The adjoint operator of $\tr_{\cA}$ is 
	\begin{equation}
		\label{eq:partial-trace-ajdj}
		\begin{array}{ccccl}
			\tr_{\cA}^{\dagger} & : & L(\cB) &\to& L(\cA \otimes \cB) \\
			& & B & \mapsto & \tr^{\dagger}_{\cA}(B) = \id_\cA \otimes B.
		\end{array}
	\end{equation}

	We denote by $T(\cA, \cB)$ the space of linear maps  
	$\Phi : L(\cA) \to L(\cB)$, known as superoperators. There is a well-known bijection between $T(\cA, \cB)$ and  $L(\cB \otimes \cA)$ given~by
	\begin{equation}
		\label{eq:CJ-map}
		\begin{array}{ccccl}
			J & :& \ T(\cA, \cB) &\to& L(\cB \otimes \cA) \\
			& & \Phi &\mapsto& J(\Phi) = \underset{i \in [n],j \in [m]}{\sum} \Phi(E_{i,j}) \otimes E_{i,j} ,
		\end{array}
	\end{equation}
	where $E_{i,j}=e_ie_j^\dagger$ 
 is the operator that sends the canonical basis element $e_j$ to the 
	canonical basis element  $e_i$.
	The corresponding matrix $J(\Phi)$  is  called the Choi representation  of $\Phi$, see,
	e.g., \cite{watrous}. Moreover,  the action of $\Phi\in T(\cA, \cB)$ (on an element $A \in L(\cA)$) can be recovered from its   Choi representation  $J(\Phi)$ through
	\begin{equation}
		\label{eq:CJ-reverse-map} 
		\Phi(A) = \tr_{\cA}\left( J(\Phi) \, (\id_{\cB} \otimes A^{\top}) \right),
	\end{equation}
	e.g., see \cite[Equation 2.66]{watrous}.
	We can express many properties of  a superoperator $\Phi$ in terms of the Choi representation~$J(\Phi)$.  The following two properties, e.g., see Theorem~2.25 and  Corollary~2.27 in \cite{watrous},
	are related to our study:
	\begin{itemize}
		\item   $J(\Phi)$ is a Hermitian matrix if and only if 
		$\Phi : L(\cA) \to L(\cB)$ is Hermitian preserving, i.e., $\Phi(X)$ is Hermitian whenever  $X$ is Hermitian.
		\item   $J(\Phi)$ is PSD if and only if $\Phi$ is completely positive, i.e.,  for  any integer $k\ge 1$ the map 
		$\mathbb{I}_k\otimes \Phi: L(\mathbb{C}^k\otimes \cA)\to L(\mathbb{C}^k\otimes \cB)$ maps  PSD matrices  to PSD matrices,
		where $\mathbb{I}_k$ is the identity map on the $k$-dimensional
		complex Euclidean space~$\mathbb{C}^k$. 
	\end{itemize}

	\subsection{Semidefinite programming}We consider complex Euclidean spaces $\cA, \cB$,  operators $C\in \her(\cA), B\in \her (\cB)$ and a Hermitian preserving superoperator  $\Phi\in T(\cA, \cB)$. The triple $(C,B,\Phi)$  specifies a   pair of primal/dual semidefinite programs~(SDPs):
	\begin{equation}\label{SDP}
		\begin{aligned}
			\sup &\  \langle C, X\rangle \\
			\text{ s.t.} &\ \Phi(X)=B, \\
			&\  X\in \pos(\cA), 
		\end{aligned} 
		\quad \quad 
		\begin{aligned} 
			\inf & \ \langle B,Y\rangle \\
			\text{ s.t.} & \ \Phi^\dagger (Y)\succeq C, \\
			& \ Y\in \her(\cB).
		\end{aligned}
	\end{equation}
Here, $\langle \cdot, \cdot \rangle$ is the Frobenius scalar product,
	$\langle A, B \rangle := \tr(A^{\dagger} \, B) = \sum_{i,j} \overline{a_{ij}} b_ {ij}$.
	SDPs represent a broad generalization of linear programming, offering significant expressive capabilities and efficient algorithms for their solution.
	The feasible region of an SDP is called a spectrahedron and the projection of a spectrahedron is called an SDP-representable set. 
	There has been significant recent interest in determining whether a convex set, e.g.,  a polyhedron, is SDP-representable, e.g., see \cite{scheiderer-2018}.
	
\section{The Model}

\subsection{Semidefinite games\label{su:sdpgames}}

  In a \emph{two-player semidefinite game} \cite{sdp-games}, 
  the strategy sets of the two players
  are described by density matrices $X_1$ and $X_2$ acting on complex 
  Euclidean spaces $\mathcal{A}_1$ and $\mathcal{A}_2$, respectively.
  Recall that a density matrix is a positive semidefinite matrix
  with trace equal to one.
The payoffs of the two players are given by the bilinear functions
  \begin{equation}
  \label{eq:bilinear1}
  u_1(X_1,X_2)=\langle R_1, X_1\otimes X_2\rangle\  \text{ and } \ u_2(X_1,X_2)=\langle R_2, X_1\otimes X_2\rangle,
  \end{equation}
  where \( R_1 \) and \( R_2 \) are Hermitian matrices acting   on the tensor product space \(\mathcal{A}_1 \otimes \mathcal{A}_2\). 
Note that as $R_i$ are  Hermitian, it is  ensured that the inner products $\langle R_i, X_1\otimes X_2\rangle$ result in real numbers.
Finally, a  strategy profile $(X_1, X_2)\in \den(\cA_1)\times \den(\cA_2)$   is a
	\emph{Nash Equilibrium}  if each player is responding optimally to the strategy of the other, i.e., 
	\begin{equation}
	\begin{aligned}
u_1(X_1, X_2) \ge u_1(X'_1, X_2), \ \forall X'_1\in \den(\cA_1) \\
\text{ and } \ u_2(X_1, X_2) \ge u_2(X_1, X'_2), \  \forall X'_2\in \den(\cA_2).
\end{aligned}
	\end{equation}

\begin{remark}
\label{rem:indices1}
Spelling out the indices in~\eqref{eq:bilinear1}, say, for a real $R$, we can write
$(R_t) = (R_t)_{ijkl}$ for $t \in \{1, 2\}$ and obtain
\[
  u_t(X_1,X_2) = 
  \sum_{\substack{
    1 \le i,j \le |\cA_1| \\
    1 \le k,l \le |\cA_2|
    }} (X_1)_{ij}  (R_t)_{ijkl} (X_2)_{kl} \, .
\]
\end{remark}

Bimatrix games can be seen as a special case of semidefinite games
(see \cite[Lemmas 4.2 and 7.1]{sdp-games}).
The entry $(i,j)$ of a payoff matrix is put into entry $(i,i,j,j)$ of
the payoff tensor and all other entries of the payoff tensor are set to 0.
Then, the mixed strategies in the bimatrix game are given by
diagonals of the density matrices in the semidefinite games. 
The non-diagonal elements of the density matrices can take any values,
as long as the positive semidefiniteness condition is satisfied. In other
words, if the players play non-diagonal density matrices $X_1$ and $X_2$,
then their payoffs are the same as with $\diag(X_1)$ and $\diag(X_2)$.

\subsection{Semidefinite network games\label{su:sdp-network-games}}		
	A \emph{semidefinite network game} is a non-coope\-ra\-tive $N$-player game taking place on an undirected graph $G=([N], E)$.  
	The strategy space of player $i\in [N]$ is the set of  
	density matrices acting on a complex Euclidean space  $\cA_i$, i.e., $\cX_i=\den(\cA_i)$. 	The set of strategy profiles $\cX = \prod_{i=1}^{N} \cX_{i}$ is the Cartesian product of the individual strategy spaces.
	Thus,  any strategy profile  $X\in \cX $ is simply a collection of density matrices, one for every node of the underlying graph, i.e.,   $X=(X_1,\ldots, X_N),$ where $X_i\in \cX_i$.
	
	In the setting of semidefinite network games, each player {\em selects a single density matrix}, which is used for playing a two-player semidefinite game  with each neighboring node.
	As typically done in network   contexts (e.g., see \cite{ccdp-polymatrix-16}),
	the payoff of each player is the cumulative payoff  acquired from these edge games.
	Concretely, in the strategy profile  $(X_1,\ldots, X_N)\in \cX$, player  $i$ is using  $X_i\in \cX_i$ to play each two-player semidefinite game corresponding to an edge $(i,j)\in E$. The  payoffs of players $i$ and $j$    in the  game corresponding  to $(i,j)\in E$ are,  respectively, given by
	$$p_{i,j}(X_i, X_j)=\langle R_{ij}, X_i\otimes X_j\rangle \ \text{ and } \ p_{j,i}(X_i, X_j)=\langle R_{ji}, X_i\otimes X_j\rangle,$$
	for some Hermitian matrices  $R_{ij} , R_{ji} \in \her(\cA_i\otimes \cA_j)$.
	 The payoff of player $i$ in the semidefinite network game is the sum of  the payoffs accrued from all the edge games in which they participate,~i.e., 
	\begin{equation*}
		p_{i}(X) = \sum_{(i,j) \in E} p_{i,j}(X_i, X_j).
		\label{eq:payoff-i}
	\end{equation*}
	
	The game is called \emph{zero-sum} if  $\sum_{i=1}^{N} p_{i}(X) = 0$, for all $X \in \cX$. Finally, a  profile $(X_1,\ldots, X_N)\in \cX$   is a 
	\emph{Nash Equilibrium}  if for each player $i$, the density matrix $X_i$ is a best response to $X_{-i}=(X_1,\ldots, X_{i-1}, X_{i+1}, \ldots, X_N)$, i.e., 
	\begin{equation}\tag{NE}
		X_i \in  \arg \max \{ p_i(Y_i,X_{-i}): Y_i\in \cX_i\}, \ \ \forall i\in [N].
	\end{equation}

	As the strategy space of each player  is convex and compact,
 	there always exists an equilibrium point, see, e.g., \cite{Glicksberg-fp-1952}. However, beyond this mere existence of an equilibrium,
 	to answer other relevant questions in semidefinite network games we need novel methods.

	\begin{example}
		Let $G=([N],E)$ be a connected graph with $|E| \ge 1$ and for any 
		edge $e \in E$ 
		let the two-player semidefinite
		game $\Gamma_e$ associated with edge $e$ be a constant-sum game.
		Then, the semidefinite network game is a constant-sum game, i.e., there
		exists a constant $C \in \R$ such that
		$\sum_{i=1}^{N} p_i(X) = C$ for every $X \in \mathcal{X}$.
		This game can be transformed into a zero-sum network semidefinite 
		game as follows.
		For any $i \in [N]$, consider the neighbors $j$ of $i$
		and construct, for $e = (i,j)$,
		from $\Gamma_e$
		the game $\Gamma'_e$ by considering
		$(p')_{i,j}(X_i,X_j) = p_{i,j}(X_i,X_j) - \frac{C}{N \deg i}$, 
		where $\deg i$ is the degree
		of the node $i$ in $G$. The game $\Gamma'_e$ can be formulated
		as a two-player semidefinite game. 
		
		The resulting network game satisfies
		\[
                      \sum_{i=1}^N p_i'(X) = \sum_{i=1}^N \sum_{(i,j) \in E}
		(p')_{i,j}(X_i,X_j)  
		= (\sum_{i=1}^N \sum_{(i,j) \in E} (p_{i,j}(X_i,X_j)) - C = 0,
                     \]
		and hence, $\Gamma'$ is a zero-sum semidefinite game.
		If $G$ is a non-connected graph with $|E| \ge 1$,
		then the normalization step to a zero-sum game
		can be carried out on an arbitrary component
		with at least one edge.
	\end{example}
	
	\begin{example}
		Let $\hat{\Gamma}$ be a zero-sum network matrix game.
		By embedding the pure strategies of the edge games into the diagonal
		entries of the payoff tensors of a corresponding semidefinite game, we obtain
		a zero-sum semidefinite game $\Gamma$. To each edge game of $\Gamma$,
		adding a two-player zero-sum semidefinite game preservers the property
		of a zero-sum semidefinite game. 
		
		For example, let $\hat{\Gamma}$ be the
		following security game considered in \cite{CaiDas-multi-NE-11}. 
		Let $G$ be a complete bipartite graph, where the nodes refer to
		evaders and inspectors. Every evader and every inspector can choose
		one of several given exits, which correspond to the pure strategies. 
		If an evader's exit is not chosen by an inspector, then the evader
		obtains one unit. For every evader whose exit is inspected,
		the inspector obtains one unit. This is a constant-sum network matrix
		game, which can be turned into a zero-sum game and viewed as
		a semidefinite network game as explained before. 
		Adding to each payoff tensor 
		of an edge game the payoff tensor of an arbitrary zero-sum semidefinite
		game with the same dimensions retains the zero-sum property
		of the semidefinite game.
	\end{example}

	\section{Nash equilibria in semidefinite network games }
	
	In this section, we generalize the computation 
	of the Nash equilibria in polymatrix games from \cite{ccdp-polymatrix-16} 
	to the case of semidefinite network games. Denoting the pure strategy sets
	of the polymatrix game by $S_1, \ldots, S_N$ and the Cartesian product
	of all mixed strategies of $\Delta$, they showed: 
	
	\begin{proposition}
              \label{pr:nash-polymatrix1}
		Consider a zero-sum polymatrix game $G$.
		If $(y,w)$ is an optimal solution to the LP
		\begin{equation}
		  \label{eq:polymatrix-lp1}
		  \begin{array}{rcl}
		  \multicolumn{3}{l}{\min_{y,w} \sum_{i \in [N]} w_i} \\
		  \text{s.t. }  w_i & \ge & u_i(s,y_{-i}) \; \ \text{ for } 
		  s \in S_i, \,
		  i \in [N], \\ 
		   y & \in & \Delta
		  \end{array}		    
		\end{equation}
		then $y$ is a Nash equilibrium of $G$. Conversely, if 
		$y$ is Nash equilibrium of $G$, then there is a $w$ such that
		$(y,w)$ is an optimal solution of the 
		LP~\eqref{eq:polymatrix-lp1}.
	\end{proposition}

	In \cite{ccdp-polymatrix-16}, two proofs for this result are given: one proof based on a best response
	characterization and Nash's Theorem and one proof based on LP duality. For our generalization of 
	Proposition~\ref{pr:nash-polymatrix1} to the semidefinite network case, the proof based
	on LP duality is the relevant one.
	Furthermore, in \cite{ccdp-polymatrix-16} it was shown that recognizing whether
	a given polymatrix game is zero-sum can be done in polynomial time. The generalization
	of this recognition result to the semidefinite network case appears in Section~\ref{se:recognize}.

\subsection{Operator representation of semidefinite network games}
	
	To prepare compact notation for the generalization of Proposition~\ref{pr:nash-polymatrix1} to the
	semidefinite network
	setting, we start from a useful operator representation 
	of the bilinear payoff function, corresponding to the semidefinite game between players 
	$i$ and $j$. 
    
	\begin{lemma}\label{payoffwithphi}
		Consider  the semidefinite game between players 
	$i$ and $j$ with payoff functions 
		$$p_{i,j}(X_i,X_j)=
		\langle R_{ij},X_i\otimes X_j \rangle,$$
		where $R_{ij}\in \her(\cA_i\otimes \cA_j)$. 
		Then, there exist linear maps  
  $\Phi_{ij}: L(\cA_j)\to L(\cA_i)$, so~that 
  	\begin{equation}\label{payoffoperaator}
			p_{i,j}(X_i,X_j)=\langle X_i, \Phi_{ij}(X_j) \rangle.
		\end{equation}
		We refer to the  maps $\Phi_{ij}$ as the payoff operators.
	\end{lemma} 

\begin{remark}
The introduction of $\Phi_{ij}$ facilitates a compact notation for the semidefinite network setup.
In view of Remark~\ref{rem:indices1}, say, for a real $R_{ij}$, we have $R_{ij} = (R_{ij})_{stkl}$ 
and an explicit description for~\eqref{payoffoperaator} would then be
\[
  p_{i,j}(X_i,X_j) = \left\langle X_i, \left( \sum\nolimits_{1 \le k,l \le |\cA_j|} (R_{ij})_{stkl} (X_j)_{kl} \right)_{1 \le s,t\le |\cA_i|} \right\rangle \, .
\]
\end{remark}

	\begin{proof} 
 Let $\Psi_{ij}: L(\cA_j)\to L(\cA_i)$ be  
   the linear map whose Choi  representation is $R_{ij}$. Note that
		$$ 
		\begin{aligned}
			p_{i,j}(X_i,X_j)& = \langle R_{ij},X_i\otimes X_j \rangle=\langle X_i\otimes \id_{\cA_j} ,R_{ij}(\id_{\cA_i}\otimes {X_j})\rangle\\
			&=\langle \tr_{\cA_j}^\dagger(X_i),R_{ij}(\id_{\cA_i}\otimes {X_j})\rangle= 	\langle X_i,\tr_{\cA_j}(R_{ij}(\id_{\cA_i}\otimes {X_j}))\rangle \\
                                &= \langle X_i, \Psi_{ij}(X_j^\top) \rangle,\end{aligned} 
		$$
		where the second equality follows by the cyclicity of trace, the third equality from~\eqref{eq:partial-trace-ajdj}, and  the last equality follows from \eqref{eq:CJ-reverse-map}.  
   Setting $${\Phi}_{ij}: L(\cA_j)\to L(\cA_i), \quad X_j \mapsto \Psi_{ij}(X^\top_j),$$
   the proof is concluded. 
	\end{proof} 
	
	For each player $i$,  let  $N(i)$ be the open neighborhood of $i$ in the network. Moreover, for each strategy profile $X=(X_1,\ldots, X_N)\in \cX$ we set $X_{N(i)}=(X_j: j\in N(i))$ as the vector of strategies of all neighbors of $i$. To keep our derivations compact, we   define the linear map
	$\Phi_{i} : \underset{j\in N(i)}{\bigoplus} L(\cA_j)\to L(\cA_i)  $
	where
	\begin{equation}\label{phimap}
		\Phi_{i}(X_{N(i)})= \sum_{j \in N(i)}  \Phi_{ij}(X_j).
	\end{equation}
	Although each $\Phi_i$ only depends on the density matrices  chosen by the neighbors of $i$, it will be convenient to view $\Phi_i$  also as  a function from $\domain$.
	Combining Lemma~\ref{payoffwithphi} with the definition of $\Phi_i$,  we can write the payoff of player $i$  as a linear function of their strategy $X_i$ explicitly, i.e., 
	\begin{equation}\label{utilitywithphi}
		p_i(X)  = \sum_{j: (i,j) \in E} p_{i,j}(X_i, X_j)
		= \sum_{j: (i,j) \in E} \langle X_i, \Phi_{ij}(X_j) \rangle 
		= \langle X_i, \Phi_{i}(X_{N(i)}) \rangle.
	\end{equation}
	Finally, we define the direct  sum of the linear maps $\Phi_i$, i.e., 
	\begin{equation}\label{gameoperator}
		\Phi: \domain \to \domain \ \text{ where }  (X_1, \ldots, X_N)\mapsto  (\Phi_1(X),\ldots, \Phi_N(X)).
	\end{equation}
	As the game is zero-sum,  for each strategy profile $X=(X_1,\ldots, X_N)\in \cX$ we have  
	\begin{equation}\label{sumpayoffs}
		\sum_ip_i(X)=\sum_i\langle X_i, \Phi_i(X)\rangle=0.
	\end{equation}
	Combining \eqref{gameoperator} and  \eqref{phimap}, \eqref{sumpayoffs}  can be also written as 
	\begin{equation}\label{zerosum}
		\langle X, \Phi(X)\rangle =0 \text{ for all } X\in \cX.
	\end{equation}
	Next, set   $\Psi=\Phi^\dagger$ and note that
	\begin{equation}\label{psimap}
		\begin{aligned} 
			& \Psi: \domain \to \domain, \quad
			(Y_1,\ldots, Y_N)\mapsto (\Psi_1(Y), \ldots, \Psi_N(Y)),
		\end{aligned} 
	\end{equation}
           where $\Psi_i(Y)=\sum_{j\in N(i)}\Phi_{ji}^\dagger (Y_j)$.
	Indeed, on the one hand we have 
	$$\langle \Psi(Y_1,\ldots, Y_N), (X_1,\ldots, X_N)\rangle=\sum_i\langle  \Psi_i(Y), X_i\rangle, $$
	and on the other hand
	\begin{eqnarray*}
		& & \langle \Psi(Y_1,\ldots, Y_N), (X_1,\ldots, X_N)\rangle =\langle (Y_1,\ldots, Y_N), \Phi(X_1,\ldots, X_N)\rangle\\
		& = & \sum_i\langle Y_i, \Phi_i(X)\rangle 
		= \sum_i\sum_{j \in N(i)} \langle Y_i, \Phi_{ij}(X_j)\rangle 
		=  \sum_i \langle  \sum_{j\in N(i)} \Phi^\dagger_{ji} (Y_j), X_i\rangle.
	\end{eqnarray*}
	Finally, as $\Psi=\Phi^\dagger$, \eqref{zerosum}  implies  that
	\begin{equation}\label{sumzero}
		\langle \Psi(X), X\rangle =0 \text{ for all } X\in \cX.
	\end{equation}
	
	With regard to the Nash equilibria 
	(as introduced in Section~\ref{su:sdp-network-games}),
	we next show that we can assume  that the payoff operators  are completely~positive.
	\begin{lemma}\label{nonnegativity}
		For any semidefinite network game there exists another semidefinite network game with the same Nash equilibria and completely  positive  payoff operators~$\Phi_{ij}$. 
		
	\end{lemma} 
	\begin{proof}
		Consider a new semidefinite network game where we replace each payoff matrix $R_{ij}$ with $R_{ij}+c\id_{ij}$, where $\id_{ij}$ denotes the
		identity operator on $\cA_i\otimes \cA_j$ and
		$c$ is chosen such that $R_{ij}+c\id_{ij}\succeq 0$ for all $(i,j)\in E$.
		This ensures that the resulting operators are positive semidefinite.   In this modified game, the payoffs for each edge are nonnegative and differ by a constant $c$ from the payoffs in the original game, i.e., $$\langle R_{ij}+c\id_{ij}, X_i\otimes X_j\rangle =\langle R_{ij}, X_i\otimes X_j\rangle +c.$$
		Thus, the payoff of player $i$ in the new game differs by  $c\deg(i)$ compared to the original payoff.  Consequently, 
		the Nash equilibria in the old game and the new game~coincide. 
		
		Finally, in the new game, the payoff operator $\Phi_{ij}$ is the linear map whose Choi representation is $ R_{ij}+c\id_{ij}$. As the latter matrix is PSD, it follows that $\Phi_{ij}$ is completely  positive. 
	\end{proof} 
	
\subsection{Semidefinite representability of Nash equilibria
   in semidefinite network games}
	
	Recall that in a semidefinite network game, 
	a  strategy  profile $(X_1,\ldots, X_N)\in \cX$  is a 
	Nash Equilibrium (NE) if for each player $i$, the density matrix $X_i$ is a best response to $X_{-i}$. For any profile $X=(X_1,\ldots, X_N)\in \cX$ and
	$i \in N$ define $e_i(X)$ as the maximum possible gain that they can achieve by deviating    from $X_i$ while the strategies of the other
	players remain fixed, i.e., 
	$$e_i(X)=\max_{S_i\in \cX_i}\langle S_i, \Phi_i(X_{N(i)})\rangle-\langle X_i, \Phi_i(X_{N(i)})\rangle.$$ 
	$e_i(X)$ is known as the \emph{Nikaido-Isoda gap} \cite{nikaido-isoda-1963}, or {\em exploitability} \cite{exploitability}.
	Then, for any strategy profile  $X=(X_1,\ldots, X_N)\in \cX$  we have that
	$$X\text{ is a NE}  \Longleftrightarrow e_i(X)=0 \text{ for all } \ i\in [N].$$
	Finally, noting that the Nikaido-Isoda gap is always nonnegative, we have that 
	$$X\text{ is a NE}  \Longleftrightarrow \sum_ie_i(X)=0 \text{ for all } \ i\in [N].$$
	Next, define  
	$$w_i(X)=\max_{S_i\in \cX_i}\langle S_i, \Phi_i(X_{N(i)})\rangle.$$
	Recall that  for a  zero-sum game we have  $\langle X, \Phi(X)\rangle=0$ (see \eqref{zerosum}), and consequently, for any $X\in \cX$ we have
	$$\sum_i e_i(X)=\sum_iw_i(X).$$
	Putting everything together, for a zero-sum semidefinite network game we have that 
	\begin{equation}\label{exploitability}
		X\text{ is a NE}  \Longleftrightarrow \sum_iw_i(X)=0 \text{ for all } \ i\in [N].
	\end{equation}
	This discussion leads to the following  characterization  of Nash equilibria in the zero-sum case.

	\begin{theorem}
               \label{th:nash-network1}
		Consider  a zero-sum semidefinite network game  $\Gamma$ with the payoff operators  $\{\Phi_{ij}\}_{i,j\in [N]}$  and let ${\rm NE}(\Gamma)$ be the set of Nash equilibria.   Let  $P^*(\Gamma)$  be the set of optimal solutions
		of the following semidefinite program\begin{equation}\label{primal}\tag{$P(\Gamma$)}
			\arraycolsep=1.4pt\def\arraystretch{1.3}
			\begin{array}{lrcl}
				\underset{w_i, X_i}{\min}     & \underset{i \in [N]}{\sum} w_{i}  \\
				{\rm s.t.} & w_{i}\id_{\cA_i} & \succeq & \Phi_i(X_{N(i)}), \\
				& \tr(X_i)&=&1, \\
				& X_i &\succeq & 0, \\
				& w &\in & \R^N, \: i \in [N], 
			\end{array}
		\end{equation}
		where  
		$\Phi_{i}(X_{N(i)})= \sum_{j \in N(i)}  \Phi_{ij}(X_j)$.  Then, the set of Nash equilibria is equal to the coordinate projection of the spectrahedron  $P^*(\Gamma)$ onto the $X$ coordinates, i.e., 
		$${\rm NE}(\Gamma)=\Big\{ X\in \cX: \exists w\in \R^N \text{ for which } (w,X)\in P^*(\Gamma)\Big\}.$$
		In particular,  the set of Nash equilibria of a zero-sum semidefinite network game is SDP-rep\-resen\-table. 
	\end{theorem}

Theorem~\ref{th:nash-network1} provides a generalization of the result for polymatrix games from Proposition~\ref{pr:nash-polymatrix1}.
To establish the generalization of the LP of this proposition, it is beneficial to work on the dual problem. The subsequent proof
employs a dual semidefinite problem (called $D(\Gamma)$ in the proof) which is the semidefinite version of the dual LP employed 
in the duality-based proof in~\cite[Theorem 1]{ccdp-polymatrix-16}, called DLP there.

	\begin{proof}
		First we consider the optimization problem 
		\begin{equation}\tag{$P'(\Gamma$)}
			\label{eq:poly-sdg-opt-primal}
			\arraycolsep=1.4pt\def\arraystretch{1.8}
			\begin{array}{lll}
				\underset{w,X}{\min}        & \underset{i \in [N]}{\sum} w_{i} \\
				\text{s.t.} & w_{i}  \geq p_{i}(S_i,X_{N(i)} \quad \text{for all } S_i \in \mathcal{X}_i, \, i \in [N], \\
				& X \in \cX, \, w=(w_1, \ldots, w_N) \in \R^N .
			\end{array}
		\end{equation}
		The proof has two main steps: First, we  show that the optimal value of \eqref{eq:poly-sdg-opt-primal} is equal to zero and second, that it can be reformulated as the SDP in \eqref{primal}. The fact that optimal solutions correspond to Nash equilibria then follows immediately by~\eqref{exploitability}.
		
		We claim that \eqref{eq:poly-sdg-opt-primal} can be expressed as
		the semidefinite program \eqref{primal}.
		Using~\eqref{utilitywithphi} we can rewrite \eqref{eq:poly-sdg-opt-primal}
		in terms of the linear maps $\Phi_i$ \eqref{phimap} as
		\begin{equation}\label{formulation}
			\arraycolsep=1.4pt\def\arraystretch{1.8}
			\begin{array}{lll}
				\underset{w,X}{\min}         & \underset{i \in [N]}{\sum} w_{i} \\
				\text{s.t.} & w_{i}  \geq 
				\langle S_i, \Phi_{i}(X_{N(i)}) \rangle, 
				& \; \forall S_i \in \mathcal{X}_i, \, i \in [N], \\
				& X \in \cX, \, w=(w_1, \ldots, w_N) \in \R^N.
			\end{array}
		\end{equation}
		Clearly, for each player $i\in [N]$ we have that 
		$ w_{i}  \geq 
		\langle S_i, \Phi_{i}(X_{N(i)}) \rangle$ for all $S_i \in \mathcal{X}_i$ if and only if
                     \[
                        w_i\ge \underset{S_i}{\max}\Big\{    \langle S_i, \Phi_{i}(X_{N(i)}) \rangle: \tr(S_i) = 1, \ S_i \succeq 0\Big\}.
                     \]
		Moreover, using \eqref{lambdamax}, for  each of the inner optimization problems we see
		\begin{equation}\label{primallmax}
			\max_{S_i}\Big\{  \langle S_i, \Phi_{i}(X_{N(i)}) \rangle \,:\, \tr(S_i) = 1, S_i \succeq 0 \Big\}=\lambda_{\max}(\Phi_{i}(X_{N(i)})).
		\end{equation}
		Finally,  using  \eqref{primallmax} and noting that
		$$\lambda_{\max}(\Phi_{i}(X_{N(i)}))\le w_i \iff \Phi_{i}(X_{N(i)}) \preceq  w_i\id_{\cA_i},$$ 
		the optimization problem  \eqref{formulation} gives the 
		semidefinite program~\eqref{primal}.
		
		Next we show that  the value of the  primal \eqref{primal} is lower bounded by zero. Indeed, every feasible
		solution $(X_i,w_i)$ of the primal satisfies
		\[
		\sum_{i=1}^N w_i \ \ge \
		\sum_{i=1}^N p_i(X_i, X_{N(i)}) \ = \ 0,
		\]
		where the last equality follows from
		the zero-sum property. Moreover, the primal is strictly feasible, since by considering   large enough $w_i$, we can make the semidefinite inequality become strict. Consequently, by strong duality for linear conic programs, e.g., see \cite[Theorem 3.4.1]{sdpbook}, 
		it follows that we have strong duality and the value of the dual is attained. 
		
		To derive the dual of \eqref{primal}, we consider the Lagrangian function, 
		\begin{align*}
			\mathcal{L}(w,X,\Lambda_i,\Lambda_i',\lambda_i) = &
			\sum_{i \in [N]} w_{i}
			+\sum_{i \in [N]} \langle  \Lambda_i,  \Phi_i(X_{N(i)})-w_iI_i\rangle 
			-\sum_{i \in [N]}\langle \Lambda_i', X_i\rangle \\
			& +\sum_{i \in [N]}\lambda_i(1-\langle I_i,X_i\rangle),
		\end{align*}
		where $I_i$ denotes the identity matrix of order $i$.
		Using the definitions of $\Phi_i$ (cf. \eqref{phimap}),  $\Psi_i$ (cf. \eqref{psimap}) and basic properties of the trace and inner product,
		the Lagrangian can be arranged as
		\begin{align*}
			\mathcal{L}(w,X,\Lambda_i,\Lambda_i',\lambda_i)
			& = \sum_i w_i(1-\tr(\Lambda_i))
			+\sum_i\sum_{j\in N(i)}\langle  \Lambda_i, \Phi_{ij}(X_j) \rangle \\ 
			& -\sum_i\langle \Lambda_i', X_i\rangle 
			 \quad +\sum_i\lambda_i(1-\langle I_i,X_i\rangle). \\
		\end{align*}
		Since 
		$\sum_{j \in N(i)}\langle  \Lambda_i, \Phi_{ij}(X_j) \rangle
		   = \sum_{j\in N(i)}\langle \Phi_{ij}^{\dagger}( \Lambda_i), X_j \rangle
	       = \sum_{j\in N(i)}\langle  X_j, \Psi_{ji}(\Lambda_i) \rangle $ 
		$= \langle  X_i, $ $\Psi_{i}(\Lambda) \rangle$,
		we obtain
		\begin{multline*}
			\mathcal{L}(w,X,\Lambda_i,\Lambda_i',\lambda_i)
			= \sum_i w_i(1-\tr(\Lambda_i))
			+\sum_i \langle  X_i, \Psi_{i}(\Lambda) - \Lambda_i' - \lambda_i I_i \rangle 
			+\sum_i\lambda_i .
		\end{multline*}
		
		Putting everything together, the dual of \eqref{primal} is given by
		\begin{equation}\label{dual}\tag{$D(\Gamma)$}
			\begin{aligned}
				\max_{\lambda_i, \Lambda_i}   & \    \sum_i \lambda_i  \\
				\text{s.t.}  & \ 
				\tr(\Lambda_i)  =  1, \\
				&    \ \Psi_{i}(\Lambda)  \succeq   \lambda_i \mathbb{I}_i, \\
				&  \ \Lambda_i  \succeq  0.
			\end{aligned}
		\end{equation}
		The program \eqref{dual}  is strictly feasible, because for given $\Lambda_i$, choosing
		the numbers $\lambda_j$ negative and with sufficiently
		large absolute values gives feasible solutions.

		We next show that \eqref{dual} has a nonpositive objective
		value. Indeed, let $\lambda_i, \Lambda_i$ be a dual feasible solution. By $\Psi_{i}(\Lambda) \succeq  \lambda_i \mathbb{I}_i$  and \eqref{lambdamax} we have that 
		\begin{equation}\label{check}
			\lambda_i\le \underset{S_i\in \cX_i}{\min}   \langle \Psi_{i}(\Lambda_{N(i)}) , S_i\rangle .
		\end{equation} 
		This implies
		$$
		\begin{aligned}
			\sum_i\lambda_i& \le \sum_i\underset{S_i\in \cX_i}{\min}   \langle \Psi_{i}(\Lambda_{N(i)}) , S_i\rangle  \\
			& =\min_{(S_1,\ldots,S_N)\in \cX}\langle (\Psi_1(\Lambda), \ldots, \Psi_N(\Lambda)), (S_1,\ldots, S_N)\rangle \\
			&=\min_{(S_1,\ldots,S_N)\in \cX} \langle \Psi(\Lambda), (S_1,\ldots, S_N)\rangle \\
			& \le  \langle \Psi(\Lambda), \Lambda \rangle =0 ,
		\end{aligned},$$
		where the first inequality follows by summing \eqref{check}, the second equality since the optimization is separable in the $S_i$'s, the third equality from the definition of the $\Psi$ map  in \eqref{psimap}, the fourth inequality as  $\Lambda_i\in \cX_i$, and the final equality from  \eqref{sumzero}.

		Putting everything together, \eqref{primal} and \eqref{dual} are a pair of primal-dual SDPs, where strict duality holds, both are attained, and the (common) value is equal to~0. 
	\end{proof}

\section{Recognizing zero-sum semidefinite network games\label{se:recognize}}    

We address the question of recognizing whether a semidefinite network game is zero-sum. 
It is not immediately clear how to do so, as the zero-sum condition $\sum_{j=1}^{N} p_{j}(X) = 0$ needs to hold for all (infinitely many) strategy
profiles $X\in \cX$. In the special case where each edge game is a matrix game,
	the situation can be reduced to the finite set of pure strategies. This was studied in 
	\cite{ccdp-polymatrix-16} where it was shown that recognizing the network version of matrix games can be
	achieved by solving a finite number of linear programs and that the number of linear programs
	is polynomial in the number of players and strategies. Hence, recognizing these games can be done in polynomial time.
	
	In the general case of semidefinite network games, we show 
	we can recognize the zero-sum property by deciding whether a certain finite set of semidefinite programs all have optimal value zero. 
	
	For every player $i$, a strategy $X_i \in \mathcal{X}_i$ and $X_{-i} \in \mathcal{X}_{-i}$, 
	let 
	$$W(X_i,X_{-i})= \sum_{j \in [N]} p_j(X_i,X_{-i}).$$
	Note that a game is constant-sum if and only if
	$$W(X_i,X_{-i})=W(Y_i, X_{-i}) \text{ for all } X_i,Y_i\in \cX_i.$$
    That is because a game is constant-sum if and only if the sum over all payoffs remains constant whenever any player switches to any other strategy.
 
	\begin{theorem}Consider a semidefinite network game $\Gamma$ with the payoff operators $\{\Phi_{ij}\}_{i,j\in [N]}$. 
		The  game $\Gamma$ is constant-sum  if and only if for all $i \in [N]$, the semidefinite program
		\begin{equation}
			\label{eq:sdp-recognizing1}
			\begin{array}{rrcll}
				& \underset{X_i,Y_i,  w_\ell}{\min} \ 
				\underset{\ell \in N(i)}{\sum} w_\ell  \\
				& (\Phi^\dagger_{i\ell}+\Phi_{\ell i})(X_i-Y_i) &  \preceq &  w_\ell \id_l, & (\ell \in N(i)), \\ [1ex]
				& w_\ell & \in & \R, & (\ell \in N(i)), \\ [1ex]
				& X_i, Y_i & \in & \mathcal{X}_i & 
			\end{array}
		\end{equation}
		attains its minimum at zero. 
	\end{theorem}
	
	If the game is constant-sum, then we can check whether the game is zero-sum by evaluating
	the sum of the payoffs at an arbitrarily chosen strategy profile of the players.
	
	\begin{proof}
		First we claim that the
		the game $\Gamma$ is constant-sum  if and only if for all $i \in [N]$
		and all
		$X_i, Y_i \in \mathcal{X}_i$, the optimization problem 
		\begin{equation}
			\label{eq:recognizing1}
			\max \Big\{ W(X_i, X_{-i}) - W(Y_i, X_{-i}): X_{-i} \in \mathcal{X}_{-i} \Big\}
		\end{equation}
		attains its maximum at zero. 
		
		Assume that the game $\Gamma$ is constant-sum. Then, by the preceding discussion we have that $ W(X_i,X_{-i})=W(Y_i, X_{-i})$ for all $ X_i,Y_i\in \cX_i $. This means that every feasible solution has value 0, so~\eqref{eq:recognizing1} has optimal value zero. 
		
		Conversely, if the optimal value of~\eqref{eq:recognizing1} is zero,
		then $W(X_i,X_{-i}) = W(Y_i,$ $X_{-i})$ for all 
		$i \in V$ and $X_i, Y_i \in \mathcal{X}_{_i}$,
		$X_{-i} \in \mathcal{X}_i$. Then $\Gamma$ is 
		a constant-sum game.   
		
		The characterization~\eqref{eq:recognizing1} involves infinitely
		many SDPs, because of the quantification over 
		$X_i, Y_i$ in the infinite set $\mathcal{X}_i$. 
		Using duality theory,
		we show that for each player $i \in [N]$,
		this characterization can be formulated
		as the SDP~\eqref{eq:sdp-recognizing1}.
		For this, note that the payoff at any edge game $(k,\ell)$ where $k,\ell\ne i$ do not appear in the objective function, as all these terms cancel out. Expanding the objective function, most terms cancel and we can
        rewrite it as
        \begin{eqnarray*}
            & & W(X_i,X_{-i})- W(Y_i,X_{-i}) \\
            & = &
            \sum_{\ell \in N(i)} \Big( p_{i,\ell}(X_i,X_{\ell}) +
            p_{\ell,i}(X_\ell,X_i) \Big)
            - \sum_{\ell \in N(i)} \Big( p_{i,\ell}(Y_i,X_{\ell}) +
              p_{\ell,i}(X_\ell,Y_i) \Big) \\
            & = & \sum_{\ell \in N(i)}
            \langle X_\ell, (\Phi^\dagger_{i\ell} + \Phi_{\ell i})(X_i) \rangle
            - \sum_{\ell \in N(i)}\langle X_\ell , (\Phi^\dagger_{i\ell}+\Phi_{\ell i}) Y_i \rangle \\
            & = & \sum_{\ell \in N(i)}\langle X_\ell , (\Phi^\dagger_{i\ell}+\Phi_{\ell i})(X_i-Y_i)\rangle.
        \end{eqnarray*} 
		So \eqref{eq:recognizing1} can be written as
		\begin{eqnarray*}
			&  & \max_{X_\ell \in \mathcal{X}_\ell} \sum_{\ell \in N(i)} \langle X_\ell , (\Phi^\dagger_{i\ell}+\Phi_{\ell i})(X_i-Y_i)\rangle 
			=  \sum_{\ell \in N(i)} \lambda_{\max} (\Phi^\dagger_{i\ell}+\Phi_{\ell i})(X_i-Y_i) \\  
			& = & \sum_{\ell \in N(i)} \min_{w_l \in \R} \big \{ w_l : \, (\Phi^\dagger_{i\ell}+\Phi_{\ell i})(X_i-Y_i) \preceq  w_l\id_l \big \} .
		\end{eqnarray*}
		For a fixed player $i$, this is equivalent to the 
		SDP~\eqref{eq:sdp-recognizing1}.
	\end{proof}
	
	\begin{remark}
		While the optimization problems \eqref{eq:recognizing1} 
		are usually considered as tract\-able, let us point out that these problems are related to the decision problem of deciding whether an SDP has
		a feasible solution (SDFP, semidefinite program feasibility problem). 
		The complexity of SDFP in the Turing machine model is not known. 
		Either 
		$\mathrm{SDFP} \in \mathrm{NP} \,\cap\, \mathrm{co{-}NP}$ or
		$\mathrm{SDFP} \not\in \mathrm{NP} \,\cup\, \mathrm{co{-}NP}$.
		One obstacle to efficiently solving SDFP is that, 
		by an example of Khachiyan described in Ramana's 
		work \cite{ramana-1997}, exponential size 
		optimal solutions of semidefinite programs can arise.
		We note that exponential size optimal 
		solutions in an SDP can also arise if the 
		optimal value is known to be zero.
		Namely, if $\mathrm{SDP}_n$ is a family of semidefinite programs with parameter $n$
		that has exponential size optimal solutions in $n$, then one can construct a 
		family of semidefinite programs with parameter $n$, that has optimal value 0 and exponential size solutions. To this end, just duplicate $\mathrm{SDP}_n$ using a new set of
		variables, called $\mathrm{SDP}'_n$, and consider the SDP whose objective function is
		the difference of the objective function of $\mathrm{SDP}_n$ and 
		$\mathrm{SDP}'_n$ and whose feasible region is the Cartesian product of
		the feasible regions of  $\mathrm{SDP}_n$ and $\mathrm{SDP}'_n$. 
	\end{remark}

\begin{remark}
As mentioned in the Introduction, every
polymatrix game can be transformed into a payoff-equivalent
polymatrix game with constant-sum edges
(see \cite[Sections 4 and 5]{ccdp-polymatrix-16}). The construction
carries over to semidefinite network games.
Hence, every zero-sum semidefinite network game $\Gamma$
can be transformed into a payoff-equivalent
semidefinite network game in which
all edges are constant-sum semidefinite games. By
further adding constant tensors to the payoff tensors
of the edges, $\Gamma$ can be transformed into a semidefinite
network game $\Gamma'$ in which all edges are zero-sum
semidefinite games and such that $\Gamma$ and $\Gamma'$ have
the same Nash equilibria.
\end{remark}

	\section{Nash equilibria via  semidefinite linear complementarity problems and the connection to quantum games}
	\label{sec:sdp-lcp}

	It is well-established that Nash equilibria of bimatrix and network games can be equivalently formulated as solutions to linear complementarity problems over the nonnegative orthant, e.g., see \cite{vonstengel} and  \cite{howson}. In this section, we extend this result and demonstrate that Nash equilibria of semidefinite 
           network games can also be characterized as solutions to semidefinite linear complementarity problems (semidefinite LCP).
	
	A semidefinite LCP instance is defined in terms of a linear map   $L:~\her(\cA) $ $ \to \her(\cA)$ and a matrix $Q\in \her(\cA)$. The objective is to find a Hermitian matrix $X$ that satisfies the following three conditions:
	\begin{equation}\label{complementarity}
		X\succeq 0, \quad  L(X)+Q\succeq 0,  \quad  \langle X,L(X)+Q\rangle=0.
	\end{equation}
	Algorithms for solving semidefinite LCPs typically  exploit the specific  properties of the linear operator $L$ and the matrix $Q$,
	see, e.g., \cite{kojima,gowda} and the references therein. 
	
	\begin{theorem}
		The set of Nash equilibria of a semidefinite network game are the solutions to a semidefinite linear complementarity problem.
	\end{theorem}
	
	\begin{proof}Consider a strategy profile  $(X_1,\ldots, X_n)$ that is  a Nash equilibrium of a semidefinite network game. By definition of a Nash equilibrium, for  any player $i\in [N]$ we have that $X_i$ is a best response to $X_{-i}$, i.e.,
		$$X_i \in \underset{_{Y_i\in \cX_i}}{\rm argmax} \langle Y_i, \Phi_i(X)\rangle.$$
		Considering the optimization problem
		$$\max\Big\{  \langle Y_i, \Phi_i(X)\rangle: Y_i\in \cX_i\Big\},$$
		its dual is given by
		$$\min_{\lambda_i, Z_i} \Big\{\lambda_i: \lambda_i \id_{\cA_i}-\Phi_i(X) =Z_i\succeq 0\Big\}.$$
		Consequently, by strong duality for SDPs we have that  $(X_1,\ldots, X_n)$  is  a Nash equilibrium iff there exist $\lambda_1,\ldots, \lambda_N$ such that for all $i\in [N]$ we have 
		\begin{equation}\label{KKT}
			\begin{aligned}
				X_i & \succeq 0,\\
				\lambda_i\id_{\cA_i}-\Phi_i(X) & \succeq 0,\\
				\tr(X_i) & =1,\\
				\langle X_i, \lambda_i \id_{\cA_i}-\Phi_i(X)\rangle & =0.\\
			\end{aligned} 
		\end{equation}
		Lemma \ref{nonnegativity} establishes that we can assume each payoff matrix to be PSD. Therefore, the map $\Phi_i$ becomes a completely positive map, implying that $\Phi_i(X)$ is also a PSD matrix.  Thus, the generalized inequality $ \lambda_i\id_{\cA_i}-\Phi_i(X)\succeq 0 $ implies that $\lambda_i\ge 0$. 
		Now consider the semidefinite LCP 
		\begin{equation}\label{LCP}
			\begin{aligned}
				X_i & \succeq 0, \ i\in [N], \\
				\id_{\cA_i}-\Phi_i(X) & \succeq 0, \ i\in [N], \\
				\langle X_i, \id_{\cA_i}-\Phi_i(X)\rangle & =0, \ i\in [N]
			\end{aligned} 
		\end{equation}
		and note  that if $X_i,\lambda_i$ are feasible for $\eqref{KKT}$ then $\tfrac{1}{\lambda_i}X_i$ is feasible for \eqref{LCP}, and conversely, if $X$ is feasible for \eqref{LCP} then $\tfrac{X_i}{\tr(X_i)}, \lambda_i=\tfrac{1}{\tr(X_i)}$ is feasible for \eqref{KKT}. 
		Finally, to write \eqref{KKT} in the standard form \eqref{complementarity} we take 
		\begin{equation}
			X\in \pos(\oplus_i \cA_i), \quad 
			L(X)=-\oplus_i\Phi_i(X),\quad Q=\oplus_i\id_{\cA_i}.
		\end{equation} 
    This concludes the proof.
	\end{proof} 

\medskip

\noindent
{\bf Relationship to quantum games\label{se:quantum-games}.}
	We briefly put the connection of semidefinite network games to quantum games into perspective.
	Quantum game theory analyzes the  interactions of  quan\-tum-enabled agents, which  can process and exchange information  in accordance with  the laws of quantum mechanics.
	The concept of a Nash equilibrium from classical games is replaced by the notion of a  Quantum Nash Equilibrium (QNE) \cite{shengyu}.
	
In the zero-sum setting, QNEs can be efficiently computed using semidefinite programming. Specifically, \cite{watrous1} establishes this result in the general framework of quantum refereed games, where quantum-enabled players interact over multiple rounds with a referee, who  ultimately determines their payoffs based on a quantum measurement. 
\cite{jain-watrous-2009} derived an efficient  parallel algorithm based on the Matrix Multiplicative weights method to find approximate equilibrium points of non-interactive zero-sum quantum games.  
On the other hand, the problem of finding QNEs is  hard for a broad
class of quantum games, see \cite{watrous2}. 

Semidefinite games are special cases of quantum refereed games, a powerful model for studying  multi-round interactions of quantum-enabled agents \cite{watrous1}. The formalism in \cite{watrous1} describes a game played between two players, Alice and Bob, which is arbitrated by a referee. The referee’s output, after interacting with Alice and Bob for a fixed number of rounds, determines their payoffs. Quantum refereed games allow for multiple rounds of interaction, where players exchange quantum registers with the referee while potentially maintaining individual memory spaces. In \cite{watrous1}, the authors demonstrated how strategies in these types of games can be represented as feasible solutions of appropriate semidefinite programs, ensuring that in the zero-sum case a min-max value and a Nash equilibrium always exist.

The simplest possible instance of the refereed quantum games framework, which we refer to as a non-interactive quantum game, is equivalent to the framework of semidefinite games, as we now explain. In this simplest setup, the two players, Alice and Bob, each has access to a quantum system. Each player holds a quantum register of a predetermined size, which is mathematically isomorphic to a complex Euclidean space,~$\mathbb{C}^k$, for some integer $k \geq 1$. The players prepare a quantum state in their respective registers: Alice prepares a state represented by a density operator $\rho$ and Bob prepares a state represented by $\sigma$. The two registers are then sent to the referee and since the players are uncorrelated, the joint system is in the tensor product state $\rho \otimes \sigma$ (also known as an unentangled state). Consequently, the referee performs a joint measurement on the combined system. According to the postulates of quantum mechanics, this measurement is described by an observable \( R \), which is a Hermitian matrix with spectral decomposition \( R = \sum_i \lambda_i P_i \), where \( P_i \) are orthogonal projectors onto the corresponding eigenspaces. Upon measuring the state \( \rho \otimes \sigma \), the probability of outcome \( \lambda_i \), which we interpret as the players’ payoff, is given by \( \tr( \rho \otimes \sigma P_i) \). As a result, the expected payoff in this game is \( \tr(\rho \otimes \sigma R) =\langle R, \rho\otimes \sigma\rangle.\)

\section{Conclusion}

We have considered semidefinite network games,
where players
reside at graph nodes and their results hinge on the actions of
neighboring players. The player's strategies entail positive semidefinite matrices
and this makes them suitable for modeling quantum games on networks. 
When the games are zero-sum, we compute the corresponding  Nash
equilibria through semidefinite programs. Identifying a
semidefinite network game equates to ascertaining that an semidefinite program 
has value zero. In cases beyond zero-sum scenarios, Nash equilibria correspond 
to solutions of a semidefinite linear complementarity problem.

\subsection*{Acknowledgements}
The authors would like to thank the anonymous reviewers for the helpful and constructive comments. 
They would also like to thank Sylvain Sorin for his comments. 
CI, TT and ET are partially supported by the joint PROCOPE project ``Quantum games
and polynomial optimization'' of the MEAE/MESR and the DAAD (57753345).
 TT is partially supported through the DFG Priority Program "Combinatorial Synergies" (grant no.\ 539847176).
ET is partially supported the PGMO grant SOAP (Sparsity in Optimization via Algebra and Polynomials),
ANR JCJC PeACE (ANR-25-CE48-3760), 
and ANR PRC ZADyG (ANR-25-CE48-7058). 
AV   is supported by the MOE Tier 2 Grant (MOE-T2EP20223-0018), NRF  Singapore under its QEP2.0 programme (NRF2021-QEP2-02-P05), the CQT++
Core Research Funding Grant (SUTD) (RS-NRCQT-00002) and partially by Project MIS 5154714 of the National Recovery and
Resilience Plan, Greece 2.0, funded by the European Union under the NextGenerationEU 
Program.
	
\bibliographystyle{abbrv}
\bibliography{games}
	
\end{document}